\documentclass[a4paper]{article}

\usepackage{amsmath,amssymb,amsthm,amsfonts,latexsym}
\usepackage[all]{xy}
\usepackage{amscd,graphicx,stmaryrd}

 \newtheorem{thm}{Theorem}[section]
 
 \newtheorem{lem}[thm]{Lemma}
 \newtheorem{prop}[thm]{Proposition}
 \theoremstyle{definition}
 
 \theoremstyle{remark}

 \numberwithin{equation}{section}

\title{Toeplitz algebras in quantum Hopf fibrations}

\author{
{\sc Elmar Wagner\footnote{
{\it MSC2010:}  47L80;  81R50 \ \  
{\it Key Words:} Toeplitz algebras, quantum spheres, quantum Hopf fibration}
} \\
\normalsize
Instituto de F\'isica y Matem\'aticas\\
\normalsize
Universidad Michoacana de San Nicol\'as de Hidalgo, Morelia, M\'exico\\
\normalsize
e-mail: {\it elmar@ifm.umich.mx}}

\newcommand{\nc}[2]{\newcommand{#1}{#2}}
\newcommand{\rnc}[2]{\renewcommand{#1}{#2}}
\rnc{\[}{\begin{equation}}
\rnc{\]}{\end{equation}}
\nc{\wegengruen}{\end{equation}}

\nc{\U}{\mathrm{U}(1)}
\def\C{{\mathbb C}}
\def\N{{\mathbb N}}

\def\Z{{\mathbb Z}}
\def\bbK{{\mathbb K}}
\def\cop{\Delta}
\def\ot{\otimes}
\def\vare{\varepsilon}
\def\ra{\rightarrow}
\def\lra{\longrightarrow}
\def\id{\mathrm{id}}
\def\ov{\overline}

\def\CS{\mathcal{C}(\mathrm{S}^1)}
\def\OU{\cO(\mathrm{U}(1))}
\def\cO{{\mathcal O}}
\def\OSpq{\mathcal{O}(\mathrm{S}^3_{pq})}
\def\OSq{\mathcal{O}(\mathrm{S}^2_{pq})}
\def\CSpq{\mathcal{C}(\mathrm{S}^3_{pq})}
\newcommand{\lin}{\mbox{$\mathrm{span}$}} 
\newcommand{\co}{{\mathrm{co}}}
\nc{\Dq}{\cO(\mathrm{D}^2_{q})}
\nc{\lN}{\ell^2(\N_0)}
\def\T{\mathcal{T}}
\def\cK{\mathcal{K}}
\def\CSq{\mathcal{C}(\mathrm{S}^2_q)} 
\def\CSfp{\mathcal{C}(\mathrm{S}^3_q)} 
\def\Cpodl{\mathcal{C}(\mathrm{S}^2_q)}
\def\s{\sigma}\def\hs{\sigma}
\def\im{\mathrm{i}}
\nc{\trans}{\mathrm{t}} 
\nc{\ha}{\mbox{$\alpha$}}
\nc{\hb}{\mbox{$\beta$}}
\nc{\hg}{\mbox{$\gamma$}}
\nc{\hc}{\mbox{$\gamma$}}
\nc{\hd}{\mbox{$\delta$}}
\nc{\SUq}{\cO(\mathrm{SU}_q(2))}
\nc{\CSU}{\mbox{$\mathcal{C}(\mathrm{SU}_q(2))$}}
\nc{\podl}{\cO(\mathrm{S}^2_{qs})}
\def\pr{\mbox{$\mathop{\mbox{\rm pr}}$}}
\def\Mat{\mathrm{Mat}} 
\def\tr{\mathrm{tr}}
\nc{\Kn}{\mbox{$\mbox{K}_0$}}
\nc{\KN}{\mbox{$\mbox{K}^0$}}
\nc{\lZ}{\mbox{${\ell}^2(\Z)$}}
\def\fS{S}
\def\cE{\mathcal{E}}
\def\A{\mathcal{A}}
\def\H{\mathcal{H}}
\def\B{\mathcal{B}}
\def\bB{\ov{\B}} \def\CB{\ov{\B}}
\def\<{\langle}
\def\>{\rangle}
\def\pr{\mathrm{pr}}
\newcommand{\prr}{\mathrm{pr}_1} 
\newcommand{\prl}{\mathrm{pr}_0} 
\newcommand{\CT}{\T}
\newcommand{\Prpm}{\mathrm{pr}_\pm}

\newcommand{\Prr}{\mathrm{pr}_+} 
\newcommand{\Prl}{\mathrm{pr}_-} 

\usepackage{geometry}
\date{}                             

\begin{document}

\maketitle

\begin{abstract}
The paper presents applications of Toeplitz algebras in 
Noncommutative Geometry. As an example, a quantum Hopf fibration is given by 
gluing trivial U(1) bundles over quantum discs (or, synonymously, Toeplitz algebras) 
along their boundaries. The construction yields associated quantum line bundles 
over the generic Podle\'s spheres which are isomorphic to those from the 
well-known Hopf fibration of quantum SU(2). The relation between these two versions 
of quantum Hopf fibrations is made precise by giving an isomorphism in the 
category of right U(1)-comodules and left modules over the C*-algebra of 
the generic Podle\'s spheres. It is argued that the gluing construction 
yields a significant simplification of index computations by obtaining 
elementary projections as representatives of K-theory classes. 
\end{abstract}


\section{Introduction}
In Noncommutative Geometry, the Toeplitz algebra has a fruitful interpretation as 
the algebra of continuous function on the quantum disc \cite{KL}. 
In this picture, the description of the Toeplitz algebra as the C*-algebra extension 
of continuous functions on the circle by the compact operators 
corresponds to an embedding of the circle into the quantum disc. 
Analogous to the classical case, one can construct ``topologically'' non-trivial quantum spaces by taking 
trivial fibre bundles over two quantum discs and gluing them along their boundaries. 
Here, the gluing procedure is described by a fibre product in an appropriate category (C*-algebras, 
finitely generated projective modules, etc.). This approach has been applied successfully to 
the construction of line bundles over quantum 2-spheres \cite{CM1,HMS,W}
and to the description of quantum Hopf fibrations \cite{BaumHMW,CM2,HMS06,HW}. 
One of the advantages of the fibre product approach is that it 
provides an effective tool for simplifying index computations. 
This has been discussed in \cite{W} on the example of the Hopf fibration of quantum SU(2) 
over the generic Podle\'s spheres \cite{P}.  Whilst earlier index computations for quantum 2-spheres 
relied heavily on the index theorem \cite{H,HMS03}, 
the fibre product approach in \cite{W} allowed to compute the index pairing directly 
by producing simpler representatives of K-theory classes. 

The description of quantum line bundles in \cite{W} bears a striking analogy to the classical case: 
the same transition functions are used to glue the trivial bundles over the (quantum) disc along 
their boundaries. However, the link between the fibre product approach 
of quantum line bundles and the Hopf fibration 
of quantum SU(2) has been established only at a ``K-theoretic level'', i.~e., it has been shown that 
the corresponding projective modules are Murray-von Neumann equivalent. 
The present work will give a more geometrical picture of the quantum Hopf fibration. 
Analogous to the classical case, we will construct a non-trivial U(1) quantum principal bundle over 
the generic Podle\'s spheres such that the associated line bundles are the 
previously obtained quantum line bundles. Here, a quantum principal bundle is described by 
a Hopf-Galois extension (see the preliminaries). 
It turns out that our $\U$ quantum principal bundle is isomorphic to a quantum 3-sphere 
from \cite{CM2}. As an application of the fibre product approach, 
we will show that the associated quantum line bundles are isomorphic 
to projective modules given by completely elementary 1-dimensional projections 
which leads to  a significant simplification of index computations.

It is known that the Hopf fibration of quantum SU(2) 
over the generic Podle\'s spheres is not given by a Hopf-Galois extension but only by a 
so-called coalgebra Galois extension (that is, U(1) is only considered as a coalgebra). 
In the present paper, we will establish a relation 
between both versions of a quantum Hopf fibration by describing an explicit isomorphism 
in the category of right U(1)-comodules and left modules over the C*-algebra of 
the generic Podle\'s spheres. Clearly, this isomorphism cannot be turned into an 
algebra isomorphism of quantum 3-spheres since otherwise the Hopf fibration of quantum SU(2) 
over the generic Podle\'s spheres would be a Hopf-Galois extension. 

\section{Preliminaries}

\subsection{Coalgebras and Hopf algebras}                                        \label{qg}

A coalgebra is a vector space $C$ over a field $\bbK$ equipped with two linear maps 
$\cop :C\ra  C\ot C$ and $\vare : S\ra \bbK$, called the comultiplication and the counit, respectively, 
such that 
\begin{align}
 &(\cop\ot \id) \circ \cop = (\id\ot\cop ) \circ \cop,\\
 & (\vare\ot \id) \circ \cop = \id =(\id\ot\vare ) \circ \cop.
\end{align}
A (right) corepresentation of a coalgebra $C$ on a $\bbK$-vector space $V$ is a linear mapping 
$\cop_V :V\ra  V\ot C$ satisfying
\[
 (\cop_V\ot \id) \circ \cop_V = (\id\ot\cop ) \circ \cop_V, \quad 
(\id\ot\vare ) \circ\cop_V =\id.
\]
We then refer to $V$ as a right $C$-comodule. 
The corepresentation is said to be irreducible if $\{0\}$ and $V$ are 
the only invariant subspaces. A linear mapping $\phi$ between right $C$-comodules 
$V$ and $W$ is called colinear, if $\cop_W\circ \phi=(\phi\ot\id)\circ \cop_V$. 

A Hopf algebra $A$ is a unital algebra and coalgebra such that $\cop$ and $\vare$ are algebra homomorphism, 
together with a linear mapping $\kappa: A\ra A$, called the antipode, such that
\[
 m \circ (\kappa\ot \id) \circ \cop(a) = \vare(a)=m\circ (\id\ot\kappa ) \circ \cop(a), \quad 
a\in A,
\]
where $m: A\ot A\ra A$ denotes the multiplication map. 

We say that $C$ and $A$ are a *-coalgebra and a *-Hopf algebra, respectively, 
if $C$ and $A$ carry an involution such that $\cop$ becomes a *-morphism. 
This immediately implies that $\vare(x^*)= \ov{\vare(x)}$. 
A finite dimensional corepresentation $\cop_V :V\ra  V\ot A$ is called unitary, if 
there exists a linear basis $\{e_1, \ldots e_n\}$ of $V$ such that 
$\cop_V(e_i)=\sum_{j=1}^n e_j\ot v_{ji}$ and 
$\sum_{j=1}^n v_{jk}^* v_{ji}=\delta_{ki}$, where $\delta_{kj}$ denotes the 
Kronecker symbol. The elements $v_{ij}$ are called matrix coefficients. 
A Hopf *-algebra $A$ is called a compact quantum group algebra if it is 
the linear span of all matrix coefficients of irreducible 
finite dimensional unitary corepresentations. 
It can be shown that then $A$ admits a C*-algebra completion $H$ in the 
universal C*-norm (that is, the supremum of the norms of all bounded irreducible 
Hilbert space *-representations). We call $H$ also a compact quantum group and 
refer to the dense subalgebra $A$ as its Peter-Weyl algebra. 
The counit of $A$ has then a unique extension to $H$, and $\cop$ has a unique extension to 
a *-homomorphism $\cop: H\ra H\,\bar\ot\, H$, where  $H\,\bar\ot\, H$ denotes the least 
C*-completion of the algebraic tensor product. 

The main example in this paper will be $H=\CS$, the C*-algebra of continuous functions 
on the unit circle $\mathrm{S}^1$. It is a compact quantum group with comultiplication 
$\cop(f)(p,q)=f(pq)$, counit $\vare(f)=f(1)$ and antipode $\kappa(f)(p)=f(p^{-1})$. 
Note that $\cop$, $\vare$ and $\kappa$ are given by pullbacks of the group operations 
of $\mathrm{S}^1=\mathrm{U}(1)$. Let $U\in\CS$, 
$U(\mathrm{e}^{\im\phi})=\mathrm{e}^{\im\phi }$, denote the unitary 
generator of $\CS$. Then the Peter-Weyl algebra of $H$ is given by 
$\OU=\lin\{\, U^N:N\in\Z\,\}$ with $\cop(U^N)=U^N\ot U^N$, $\vare(U^N)=1$ 
and $\kappa(U^N)=U^{-1}$. Note also that the irreducible unitary 
corepresentations of $\OU$ are all 1-dimensional and are given by 
$\cop_\C(1)=1\ot U^N$.  

From the previous paragraph, it becomes clear why noncommutative compact quantum groups are regarded 
as generalizations of function algebras on compact groups. We give now the definition for a quantum analogue 
of principal bundles. First we remark that a group action on a topological space corresponds to 
a coaction of a quantum group or, more generally, to a coaction of a coalgebra. 
Now let $A$ be a Hopf algebra, $P$ a unital algebra, and 
$\cop_P : P\ra P\ot A$ a corepresentation which is also an algebra homomorphism 
(one says that $P$ is a right $A$-comodule algebra). 
Then the space of coinvariants
$$
   P^{\co A}:=\{ b\in P : \cop(b)=b\ot 1\}
$$
is an algebra considered as a function algebra on the base space, 
and $P$ plays the role of a function algebra on the total space. If $A$ is a Hopf *-algebra 
and $P$ is a *-algebra, we require $\cop$ to be a *-homomorphism so that $B$ 
becomes a unital *-subalgebra of $P$. 

If $\cop : P\ra P\ot C$ is a corepresentation of a coalgebra $C$, then we set 
$$
   P^{\co C}:=\{\, b\in P : \cop(bp)=b\cop(p)\ \, \text{for all}\ \, p\in P \,\}
$$
with multiplication $b(p\ot c)= bp\ot c$ on the left tensor factor. Again, 
$P^{\co C}$ is a subalgebra of $P$. In our examples, there will be a group like element 
$e\in C$ (that is, $\cop(e)=e\ot e$) such that $\cop(1)= 1\ot e$ and 
$$
   P^{\co C}=B:=\{\, b\in P : \cop(b)=b\ot e\, \}. 
$$
If $P$ and $C$ carry an involution, $\cop_P$ is a *-morphism and $e^*=e$, then $B$ 
is a *-subalgebra of $P$. 

Analogous to right corepresentations, one defines left corepresentations 
${}_V\cop: V\ra C\ot V$. The associated (quantum) vector bundles are given by the 
cotensor product $P\,\Box_C\, V$, where 
$$
P\,\Box_C\, V :=\{\, x\in P\ot V : (\cop_P\ot \id)(x)=(\id\ot{}_V \cop)(x)\,\}.
$$
Obviously, $P\,\Box_C\, V$ is a left $P^{\co C}$-module. 
For the 1-dimensional representation ${}_\C\cop(1)=U^N\ot 1$,
this module is equivalent to 
$$
P_N:= \{\, p\in P: \cop_P(p)=p\ot U^N\, \} 
$$
and is considered as a (quantum) line bundle.

\subsection{Pullback diagrams and fibre products}        \label{sec-fp}

The purpose of this section is to collect some elementary facts about 
fibre products. For simplicity, we start by considering the category 
of vector spaces. Let $\pi_0: A_0 \rightarrow A_{01}$ and  $\pi_1:A_1\rightarrow
A_{01}$ be vector spaces morphisms. Then the fibre product
$A:=A_0{\times}_{(\pi_0,\pi_1)} A_1$ is defined by the
pullback diagram
\begin{equation}                                   \label{A_is_fibre_product}
        \begin{CD}
    {A} @ >{\mathrm{pr}_1}>> {A_1} @.\\
    @ V{\mathrm{pr}_0} VV @ V{\pi_1} VV @.\\
    {A_0} @ >{\pi_0} >> {A_{01}} @. \ .\\
        \end{CD}
\end{equation}
Up to a unique isomorphism, $A$ is given by
\begin{equation}                                       \label{A}
A=\left\lbrace (a_0,a_1)\in A_0\times A_1 : \pi_0(a_0)=\pi_1(a_1)
\right\rbrace,
\end{equation}
where the morphisms
$\mathrm{pr}_0:A\rightarrow A_0$ and
$\mathrm{pr}_1:A\rightarrow A_1$ are the left and right projections,
respectively.
In this paper, we will consider fibre products in the following categories: 

\begin{itemize}
\item
If $\pi_0: A_0 \rightarrow A_{01}$ and  $\pi_1:A_1\rightarrow
A_{01}$ are morphisms of *-al\-ge\-bras, then the fibre product 
$A_0{\times}_{(\pi_0,\pi_1)} A_1$ is a *-al\-ge\-bra 
with componentwise multiplication and involution. 

\item
If we consider the pullback diagram \eqref{A_is_fibre_product} in
the category of unital $C^*$-algebras, then $A_0{\times}_{(\pi_0,\pi_1)} A_1$
will be a unital $C^*$-algebra. 

\item
If  $B$ is an algebra and $\pi_0: A_0 \rightarrow
A_{01}$ and  $\pi_1:A_1\rightarrow A_{01}$ are morphisms of left
$B$-modules, then the fibre product
$A:=A_0{\times}_{(\pi_0,\pi_1)} A_1$ is a left $B$-module
with left action $b.(a_0,a_1)=(b. a_0,b. a_1)$, where $b\in B$
and the dot denotes the left action. 

\item
If we consider the pullback diagram \eqref{A_is_fibre_product} in
the category of right $C$-comodules (or right $H$-comodule algebras), 
then $A:=A_0{\times}_{(\pi_0,\pi_1)} A_1$
will be a right $C$-comodule (or a right $H$-comodule algebra) with the 
coaction given by $\cop_A(a_1,a_2)=(\cop_{A_1}(a_1), 0)+ (0,\cop_{A_2}(a_2))$. 
\end{itemize}

Finally we remark that if $B_0$, $B_1$ and $B_{01}$ are dense subalgebras of 
$C^*$-algebras $A_0$, $A_1$ and $A_{01}$, respectively, and
$\pi_0$ and $\pi_1$ restrict to morphisms
$\pi_0: B_0 \rightarrow B_{01}$ and  $\pi_1:B_1\rightarrow B_{01}$, 
then $B_0{\times}_{(\pi_0,\pi_1)} B_1$ is not necessarily 
dense in $A_0{\times}_{(\pi_0,\pi_1)} A_1$. 
A useful criterion for this to happen can be found in \cite[Theorem 1.1]{HW}. 
It suffices that 
$\pi_1\!\!\upharpoonright_{B_1} : B_1 \rightarrow B_{01}$ is 
surjective and $\ker(\pi_1) \cap B_1$ is dense in $\ker(\pi_1)$.

\subsection{Disc-type quantum 2-spheres}      \label{sec-dq}

From now on we will work over the complex numbers and $q$ will 
denote a real number from the interval $(0,1)$. 

The *-algebra $\Dq$ of polynomial functions on the quantum disc
is generated by two generators $z$ and $z^{*}$ with relation
\[                                                          \label{Dq}
z^* z -q z z^* = 1- q. 
\]
A complete list of bounded irreducible *-representations of $\Dq$ can be found in
\cite{KL}. First, there is a faithful representation  on the Hilbert space $\lN$. 
On an orthonormal basis $\{ e_n:n\in\N_0\}$, the action of the generators reads as 
\[                                                           \label{z} 
 ze_n=\sqrt{1-q^{n+1}}Se_n, \quad z^*e_n=\sqrt{1-q^{n}}S^*e_n, 
\]
where 
$$
Se_n=e_{n+1}, 
$$ 
denotes the shift operator on $\lN$. 

Next, there is a 1-parameter family of irreducible *-representations $\rho_u$ on $\C$, 
where $u\in \mathrm{S}^1=\{x\in\C:|x|=1\}$. They are given by assigning 
$$
\rho_u(z)=u,\qquad \rho_u(z^*)= \bar u. 
$$
The set of these representations is considered as the boundary $\mathrm{S}^1$ of the quantum disc 
consisting of ``classical points''. 

The universal C*-algebra of $\Dq$
is well known. It has been discussed by several authors
(see, e.g., \cite{KL,MNW,s-a91}) that it
is isomorphic to the Toeplitz algebra $\T$. 
Here, it is convenient to view 
the Toeplitz algebra $\T$ as
the universal C*-algebra generated by $S$ and $S^*$ in $\mathrm{B}(\lN)$.
Then above *-representation on $\lN$ becomes simply an embedding. 

Another characterization is given by the C*-extension 
$$
0\lra \cK(\lN) \lra\T\stackrel{\sigma}{\lra}  \CS \lra  0,
$$
where $\sigma : \T \ra \CS$ is the so-called symbol map and corresponds,
in the classical case, 
to an embedding of $\mathrm{S}^1$ into the complex unit disc. Let again 
$U(\mathrm{e}^{\im\phi})=\mathrm{e}^{\im\phi }$ 
denote the unitary generator of $\CS$.  
Then the symbol map is completely determined by 
setting $\s(z)=U$. 

We can now construct a quantum 2-sphere $\CSq$ by gluing two quantum discs along their boundaries. 
The gluing procedure is described by the fibre product $\T\times_{(\s,\s)}\T$, 
where $\T\times_{(\s,\s)}\T$ is defined 
by the following pullback diagram in 
the category of C*-algebras: 
\begin{equation}                                   \label{TtimesT}
        \begin{CD}
    {\T\underset{(\s,\s)}{\times}\T} @ >{\mathrm{pr}_1}>> {\T} @.\\
    @ V{\mathrm{pr}_0} VV @ V{\s} VV @.\\
    {\T} @ >{\s} >> {\CS}. @.\\
        \end{CD}
\end{equation}
Up to isomorphism, the C*-algebra $\CSq:= \T\times_{(\s,\s)}\T$ is given by 
\begin{equation}                                           \label{CSq}
\CSq=\{\, (a_1,a_2)\in \T\times\T : \s(a_1)=\s(a_2) \,\}. 
\end{equation}

In the classical case, complex line bundles with winding number $N\in\Z$ over the 2-sphere
can be constructed by taking trivial bundles over the northern and southern hemispheres
and gluing them together along the boundary via the map
$U^{N}:\mathrm{S}^1  \ra \mathrm{S}^1$,\,
$U^{N}(\mathrm{e}^{\im\phi})=\mathrm{e}^{\im\phi N}$.
In \cite{W}, the same construction has been applied to to the quantum 2-sphere $\CSq$. 
The roles of the northern and southern hemispheres are played by two copies of the 
quantum disc, and the transition function along the boundaries remains the same. 
This construction can be expressed by the following pullback diagram:
\begin{equation}                                                         \label{TT}
\xymatrix{
& \T \,{\underset{(U^N\sigma ,\sigma)}{\times}}\, \T
\ar[dl]_{\mathrm{pr}_0} \ar[dr]^{\mathrm{pr}_1}& \\
\T \ar[d]_{\sigma} &    &  \T \ar[d]^{\sigma}\\
\CS \ar[rr]_{f\mapsto U^Nf} & & \CS.   }
\end{equation}
So, up to isomorphism, we have
\begin{equation}                                                     \label{fpLB}
 \T \times_{(U^N\sigma ,\sigma)} \T
\cong \{\,(a_0,a_1)\in \T\times \T \,:\, U^{N}\s(a_0)=\s(a_1)\,\}.
\end{equation}
It follows directly from Equation \eqref{CSq} that 
$\T \times_{(U^N\sigma ,\sigma)} \T$ is a $\CSq$-(bi)module. 
This can also be seen from the general pullback construction by 
equipping $\T$ and $\CS$ with the structure of a left $\CSq$-module. 
Explicitly, for $(a_0,a_1)\in\CSq$, one defines 
$(a_0,a_1).a= a_0 a$ for $a\in\T$ on the left side, 
$(a_0,a_1).a= a_1 a$ for $a\in \T$ on the right side, and 
$(a_0,a_1).a= \s(a_0) b=\s(a_1)b$ for $b\in\CS$. 

To determine the K-theory and K-homology of $\CSq$, 
we may use the results of \cite{MNW}.
There it is shown that $\Kn(\Cpodl)\cong \Z \oplus \Z$
and $\KN(\Cpodl)\cong \Z \oplus \Z$. 
The two  generators of the $\Kn$-group can be chosen to be the class $[1]$ of
the unit element of $\CSq$, and the class  $[(0,1-SS^{*})]$. 

Describing an even Fredholm module by a pair of representations 
on the same Hilbert space such that the difference is a compact operator,
one generator of $\KN(\Cpodl)$ is obviously given
by the class $[(\mathrm{pr}_1,\mathrm{pr}_0)]$ 
on the Hilbert space $\lN$. 
A second generator is $[(\pi_+\circ\hs\,,\,\pi_-\circ\hs)]$, 
where $\s$ denotes the symbol map and 
$\pi_\pm: \CS\ra\mathrm{B}(\lZ)$ is given by 
\begin{align}
 \begin{split}                            \label{pi}
  &\pi_+(U)e_n=e_{n+1}, \quad n\in\Z,\\
&\pi_-(U)e_n=e_{n+1},\quad n\in\Z\setminus \{-1,0\},\quad
\pi_-(U)e_{-1}=e_1, \quad
\pi_-(U)e_{0}=0 
 \end{split}
\end{align}
on an orthonormal basis $\{e_n\,:\, n\in\Z\}$ of $\lZ$. 
Note that the representation $\pi_-$ is non-unital: $\pi_-(1)$ 
is the projection onto $\lin\{e_n\,:\, n\in\Z\setminus \{0\}\,\}$.

\subsection{Quantum 3-spheres and quantum Hopf fibrations}     \label{qhf}

First we follow \cite{HMS06} and introduce the coordinate ring 
of a Heegaard-type quantum 3-sphere $\OSpq$, $p,q\in (0,1)$ as the *-algebra 
generated by $a$, $a^*$, $b$, $b^*$ subjected to the relations 
\begin{align}
\begin{split}                                      \label{ab}
 & a^*a-qaa^*= 1-q, \quad b^*b-p b b^* = 1-p, \\
 & (1-aa^*)(1-bb^*)=0,\quad ab=ba,\quad a^*b = ba^*. 
\end{split} 
\end{align}
Its universal C*-algebra (i.e., the closure of $\OSpq$ in the universal C*-norm  
given by the supremum over all bounded Hilbert space representations) will be 
denoted by $\CSpq$. 

One can easily verify that the coaction $\cop_{\OSpq}: \OSpq\ra \OSpq\ot \OU$
given by 
$$
\cop_{\OSpq}(a)=a\ot U^*,\quad \cop_{\OSpq}(b)=b\ot U
$$
turns $\OSpq$ into a $\OU$-comodule *-algebra. 
Its *-subalgebra of  $\OU$-coinvariants 
$\OSq:=\OSpq^{\co \OU}$ is generated by 
$$
   A:= 1-aa^* , \quad B:= 1-bb^*,\quad  R:= ab 
$$
with involution $A^*=A$, $B^*=B$ and commutation relations 
\begin{equation*}
 R^*R=1-qA-pB, \ \                                                           
RR^*=1-A-B,    \ \ 
AR=qRA, \ \ 
BR =pRB,\ \ 
AB=0.                                                      
\end{equation*}

Note that $\OSq$ can also be considered as a *-subalgebra 
of $\CSq$ from \eqref{CSq} by setting 
$$
A=( 1-zz^*,0), \quad B=(0,1-yy^*), \quad R=(z,y),
$$
where $y$ and $z$ denote the generators of the quantum discs 
$\cO(\mathrm{D}^2_{p})$ and $\cO(\mathrm{D}^2_{q})$, respectively, 
satisfying the defining relation \eqref{Dq}. Using the fact that 
$\cO(\mathrm{D}^2_{q})$ is dense 
in the Toeplitz algebra $\T$ for all $q\in(0,1)$, 
and the final remark of Section \ref{sec-fp}, one easily proves 
that $\CSq= \T\times_{(\s,\s)}\T$ is the universal C*-algebra 
of $\OSq$. 

For $N\in \Z$, let 
\[                                                        \label{LN}
 L_N:= \{\, p\in \OSpq\,:\, \cop_{\OSpq}(p)=p\ot U^N\,\}
\]
denote the associated quantum line bundles. It has been shown in 
\cite{HMS06} that $L_N$ is isomorphic to $\OSq^{|N|+1} E_N$, 
where
\[                                                          \label{EN}                            
E_N = X_N\, Y_N^\trans\in \Mat_{|N|+1,|N|+1}(\OSq)
\]
and, for $n\in\N$, 
\begin{align*}                                                    
& X_{-n}=(b^{*n},ab^{*n-1},\ldots,a^n)^\trans,\qquad 
X_{n}=(a^{*n},ba^{*n-1},\ldots,b^n)^\trans,\\
&Y_{-n}=\Big(\mbox{$\binom{n}{0}$}_{\!p}\,p^nA^nb^n\,,\,
\mbox{$\binom{n}{1}$}_{\!p}p^{n-1}A^{n-1}b^{n-1}a^*\,,\,\ldots\,,\,
\mbox{$\binom{n}{n}$}_{\!p}a^{*n}\Big)^\trans, \\
&Y_{n}=\Big(\mbox{$\binom{n}{0}$}_{\!q}q^nB^na^n\,,\,
\mbox{$\binom{n}{1}$}_{\!q}q^{n-1}B^{n-1}a^{n-1}b^*\,,\,\ldots\,,\,
\mbox{$\binom{n}{n}$}_{\!q}b^{*n}\Big)^\trans, 
\end{align*}
with 
$$
\mbox{$\binom{n}{0}$}_{\!x}=
\mbox{$\binom{n}{n}$}_{\!x}:=1, \quad 
\mbox{$\binom{n}{k}$}_{\!x}:=\mbox{$
\frac{(1-x)\dots (1-x^n)}{(1-x)\dots (1-x^k)(1-x)\dots (1-x^{n-k})}$}, \ \, 0<k<n, \
\, 
x\in (0,1).  
$$
That $E_N$ is indeed an idempotent follows from $Y_N^\trans \, X_N= 1$ 
which can be verified by direct computations. 

Now we consider a much more prominent example of a quantum Hopf fibration. 
The *-algebra $\SUq$ of polynomial functions on the quantum group
$\mathrm{SU}_q(2)$
is generated by $\ha$, $\hb$, $\hc$, $\hd$ with relations
\begin{align*}
& \ha \hb =q \hb \ha,\quad \ha \hc =q\hc \ha,\quad \hb \hd = q \hd \hb,\quad
\hc \hd = q \hd \hc, \quad \hb \hc =\hc \hb,\\
& \ha \hd - q \hb \hc = 1, \quad  \hd \ha - q^{-1} \hb \hc = 1,
\end{align*}
and involution $\ha^*=\hd$, $\hb^*=-q\hc$. This is actually a Hopf *-algebra with 
the Hopf structure $\Delta$, $\vare$, $\kappa$. 
Here, we will only need explicit formulas for the homomorphism 
$\vare : \SUq\ra \C$ given by 
$$
\vare(\ha)=\vare(\hd)=1,\quad \vare(\hb)=\vare(\hc)=0. 
$$

For $s\in (0,1]$, the *-subalgebra generated by 
$$                                                   
 \eta_s:=(\hd+q^{-1}s\hb)(\hb-s\hd),\quad \zeta_s:=1-(\ha-qs\hc)(\hd+s\hb).
$$
is known as the generic Podle\'s sphere $\podl$ \cite{P}. Its generators satisfy the defining 
relations 
$$                                                      
\zeta_s \eta_s = q^2 \eta_s \zeta_s,\quad \!\!
\eta_s^*\eta_s=(1-\zeta_s)(s^2+\zeta_s),\quad \!\!
\eta_s\eta_s^*=(1-q^{-2}\zeta_s)(s^2+q^{-2}\zeta_s),
$$
and $\zeta_s^*=\zeta_s$. For all $s\in (0,1]$ and $q\in(0,1)$,
the universal C*-algebra of $\podl$ is isomorphic to $\CSq$ \cite{MNW,s-a91}. 
With $x$ the generator of $\cO(\mathrm{D}^2_{q^2})$, set $t:=1-xx*\in \T$. 
An embedding of $\podl$ into $\CSq$ as a dense *-subalgebra is given by 
\[                 
\zeta_s =(-s^2 q^2 t, q^2 t), \quad                                 \label{etaT}
\eta_s=\left(s\sqrt{(1-q^2t)(1+s^2 q^2t)}\,S \,,\, \sqrt{(1-q^2 t )(s^2+q^2 t)}\,S\right). 
\]

Let $\podl^+:=\{x\in\podl\,:\,\varepsilon (x)=0\}$.  
It has been shown in \cite{MS} that the quotient space $\SUq/\podl^+\SUq$ 
with coaction $(\pr_s\ot \pr_s)\circ \Delta$
is a coalgebra isomorphic to $\OU$. Here $\pr_s$ denotes the canonical projection 
and $\Delta$ the coaction of $\SUq$. 
We emphasize that this isomorphism holds only in the category of coalgebras, 
that is, $\SUq/\podl^+\SUq$ is a linear space (not an algebra!) 
spanned by basis elements $U^N$, $N\in \Z$, with coaction 
$\Delta( U^N)= U^N\ot U^N$. The composition $(\id\ot\pr_s)\circ \Delta$ turns 
$\SUq$ into an $\OU$-comodule and the associated line bundles are given by 
$$
  M_N:=\{\, p\in \SUq\,:\, (\id\ot\pr_s)\circ \Delta (p)= p\ot U^N\,\},\quad N\in \Z. 
$$
Moreover,  $\podl=M_0=\SUq^{\co \OU}$ and $\SUq=\oplus_{N\in\Z} M_N$. 
In contrast to quantum line bundles $L_N$ defined above, $M_N$ is 
only a \emph{left} $\podl$-module but not a \emph{bi}module. 
This is also due to the fact that $\SUq$ with above coaction 
is only an $\OU$-comodule but not an $\OU$-comodule algebra. 

Explicit descriptions of
idempotents representing $M_N$ have been given in \cite{HMS03,SW}. 
Analogous to $L_N$, there are elements 
$v^N_{0},v^N_{1}, \ldots, v^N_{|N|} \in \SUq$ 
such that $M_N\cong \podl^{|N|+1}P_N$, where 
\[                                                              \label{PN}
P_N:= (v^N_{0},v^N_{1}, \ldots, v^N_{|N|})^\trans\, 
(v^{N*}_{0},v^{N*}_{1}, \ldots, v^{N*}_{|N|})\in \Mat_{|N|+1,|N|+1}(\podl) 
\]
with 
\[                                                               \label{1}
(v^{N*}_{0},v^{N*}_{1}, \ldots, v^{N*}_{|N|})\,(v^N_{0},v^N_{1}, \ldots, v^N_{|N|})^\trans=1.
\] 
For a definition of $v^N_{k}$, see \cite{SW}. 

A description of the universal C*-algebra $\CSU$ of $\SUq$ as a fibre product can be found 
in \cite{HW}. There it is shown that $\CSU$ is isomorphic to the fibre product C*-algebra 
of the following pullback diagram: 
\begin{equation}                                             \label{CSUq}
\xymatrix{
& \makebox[48pt][c]{$\T\,\bar\otimes\,\CS \,{\underset{(W\circ \sigma\bar\otimes\id ,\pi_2)}{\times}}\, \CS$}
\ar[dl]_{\mathrm{pr}_1} \ar[dr]^{\mathrm{pr}_2}& \\
\T \,\bar\otimes\,\CS\ar[d]_{\sigma\bar\otimes\id } &    &  \CS \ar[d]^{\pi_2}\\
\CS\,\bar\otimes\,\CS \ar[rr]_{W} & & \CS\,\bar\otimes\,\CS\;.   }
\end{equation}
Here, $\pi_2: \CS\ra \CS\,\bar\otimes\,\CS$ is defined by $\pi_2(f)(x,y)=f(y)$, and 
\[                                           \label{W}
 W: \CS\,\bar\otimes\, \CS\ra \CS\,\bar\otimes\, \CS,\quad W(f)(x,y)=f(x,xy),
\]
is the so-called multiplicative unitary. In the next section, we will frequently use that 
$W(g\ot U^N)(x,y)=g(x)x^Ny^N=(gU^N\ot U^N)(x,y)$, that is, 
\[                                                                \label{WU}
W(g\ot U^N) = gU^N\ot U^N
\]
for all $g\in \CS$ and $N\in\Z$.  As above,  
$U$  denotes the unitary generator of $\CS$ given by 
$U(\mathrm{e}^{\im\phi})=\mathrm{e}^{\im\phi }$  for $\mathrm{e}^{\im\phi}\in\mathrm{S^1}$.

\section{Fibre product approach to quantum Hopf fibrations}

\subsection{C*-algebraic construction of a quantum Hopf fibration}  \label{Cfp}

The aim of this section is to construct a $\mathrm{U}(1)$ quantum principal bundle 
over a quantum 2-sphere such that the associated quantum line bundles 
are given by \eqref{TT}. Our strategy will be to start with trivial $\mathrm{U}(1)$-bundles 
over two quantum discs and to glue them together along their boundaries by a non-trivial 
transition function. Working in the category of C*-algebras, an obvious quantum analogue 
of a trivial bundle $D\times \mathrm{S}^1$ is given by the completed tensor product 
$\T\,\bar\ot\, \CS$, where we regard $\T$ as the algebra of continuous functions 
on the quantum disc. Since $\CS$ is nuclear, 
there is no ambiguity about the tensor product completion. 

Recall from Section \ref{qg} that a group action on a principal bundle gets translated to 
a Hopf algebra coaction (or, slightly weaker, coalgebra coaction). 
As our group is $\mathrm{U}(1)=\mathrm{S}^1$, we take the Hopf *-algebra $\CS$ introduced in 
Section \ref{qg}. On the trivial bundle $\T\,\bar\ot\,\CS$, we consider the ``trivial'' 
coaction given by applying the coproduct of $\CS$ to the second tensor factor. 
The gluing of the trivial bundles $\T\,\bar\ot\,\CS$ will be accomplished by a 
fibre product over the ``boundary'' $\CS\,\bar\ot\,\CS$. To obtain a non-trivial
fibre bundle, we impose a non-trivial transition function. 
From the requirement that the associated quantum line bundles should be 
given by \eqref{TT}, the transition function is easily guessed:  
We use the multiplicative unitary $W$ from \eqref{W}. The result is described by 
the following pullback diagram.  
\begin{equation}                                             \label{CSfp}
\xymatrix{
& \makebox[48pt][c]{$
\T\,\bar\otimes\,\CS {\underset{(W\circ \pi_1,\pi_2)}{\times}} \T\,\bar\otimes\,\CS$}
\ar[dl]_{\mathrm{pr}_1} \ar[dr]^{\mathrm{pr}_2}& \\
\T \,\bar\otimes\,\CS\ar[d]_{\pi_1:=\sigma\bar\otimes\id } & & \T \,\bar\otimes\,\CS\ar[d]^{\pi_2:=\sigma\bar\otimes\id}\\
\CS\,\bar\otimes\,\CS \ar[rr]_{W} & & \CS\,\bar\otimes\,\CS\;. }
\end{equation}
For brevity, we set 
$\CSfp:=\T\,\bar\otimes\,\CS {\times}_{(W\circ \pi_1 ,\pi_2)} \T\,\bar\otimes\,\CS$. 
Note that $\sigma\bar\otimes\id$ and $W$ are morphisms of right $\CS$-comodule algebras. 
Thus $\CSfp$ is a right $\CS$-comodule algebra (cf.\ Section \ref{qg}) or, 
in the terminology of Section \ref{qg}, a $\CS$ quantum principal bundle. 
Its relation to the (algebraic) Hopf fibration of $\OSpq$ and to the 
quantum line bundles from Equation \ref{fpLB} will be established in the next 
proposition. 

\begin{prop}                                                                 \label{prop1}
$\CSfp$ is the universal C*-algebra of $\OSpq$, the associated quantum line bundles 
\[                                                                       \label{CN}
   \CSfp_N:=\{\, p\in \CSfp \, :\, \cop_{\CSfp}(p)=p\ot U^N\,\}, \quad N\in\Z,  
\]
are isomorphic to $\T \times_{(U^N\sigma ,\sigma)} \T$ from \eqref{fpLB}, 
and $L_N\subset \CSfp_N$. Here, $L_N$ denotes the quantum line bundle defined in \eqref{LN},  
and $U$ is the unitary generator of $\CS$. 
\end{prop}
\begin{proof}
Let $z$ and $y$ be the generators of the quantum discs 
$\cO(\mathrm{D}^2_{q})$ and $\cO(\mathrm{D}^2_{p})$, respectively. 
Consider the *-algebra homomorphism $\iota:\OSpq\ra \CSfp$ given by 
\[                                                                 \label{i}
\iota(a)= (z\ot U^*, 1\ot U^*), \quad  \iota(b)= (1\ot U, y\ot U). 
\]
Choosing a Poincar\'e-Birkhoff-Witt basis of $\OSpq$, for instance all ordered 
polynomials in $a$, $a^*$, $b$, $b^*$, 
and using the embedding $\Dq\subset \T$, one easily verifies that $\iota$ is 
injective. 
Moreover, since the operators $\pi(a)$ and $\pi(b)$ satisfy the quantum disc 
relation \eqref{Dq} for any bounded representation $\pi$, the 
*-representation $\iota$ is actually an isometry if we equip $\OSpq$ with the 
universal C*-norm. Therefore it suffices to prove that $\iota(\OSpq)$ is 
dense in $\CSfp$. 
For this, consider the  image of $\iota(\OSpq)$ under the projections 
$\pr_1$ and $\pr_2$. Since $1\ot U =\pr_1(\iota(b))\in \pr_1(\iota(\OSpq))$ and 
$z\ot 1= \pr_1(\iota(ab)) \in \pr_1(\iota(\OSpq))$, we get 
$\pr_1(\iota(\OSpq))=\Dq\ot \OU$, and similarly 
$\pr_2(\iota(\OSpq))=\Dq\ot \OU$. Note that the latter is a 
dense *-subalgebra of $\T\ot\CS$. Moreover, 
$(\hs\bar\ot\id)(\Dq\ot \OU)=\OU\ot\OU$ is dense in $\CS\,\bar\ot\,\CS$, 
and $W:\OU\ot\OU\ra \OU\ot\OU$ is an isometry. 
Since $W(U^n\ot U^m)=U^{n+m}\ot U^m$ for all $n,m\in\Z$ by \eqref{WU}, it is a bijection 
of $\OU\ot\OU$ onto itself. From the foregoing, it follows that 
$ \iota(\OSpq)= 
\Dq\otimes\OU {\times}_{(W\circ \sigma\otimes\id ,\pi_2)} \Dq\otimes\OU$. 
By considering the ideal generated by the compact operator $1-zz^*\in \Dq$ 
(or $1-yy\in\cO(\mathrm{D}^2_{p})$), one easily shows that 
$\ker(\hs\bar\ot\id)\cap (\Dq\ot \OU)$ is dense in $\ker(\hs\bar\ot\id)$. 
From the final remark in Section \ref{sec-fp}, we conclude that 
$\iota(\OSpq)$ is dense in $\CSfp$. 

To determine $\CSfp_N$, recall that the coaction is given by the 
coproduct on the second tensor factor $\CS$. Assume that 
$f\in\CS$ satisfies $\cop(f)=f\ot U^N$. Then it follows from
$f=(\vare\ot\id)\circ\cop(f)= f(1)U^N$ that  
$(\id\ot\cop)(x)=x\ot U^N$ for $x\in \T\,\bar\ot\,\CS$ if and only if 
$x=t\ot U^N$ with $t\in \T$. 
Since the morphisms in the pullback diagram \eqref{CSfp} are 
right colinear, we get $p\in \CSfp_N$ if and only if 
$p=(t_1\ot U^N, t_2\ot U^N)$ and 
$(W\circ \sigma\bar\otimes\id)(t_1\ot U^N)=(\sigma\bar\otimes\id)(t_2\ot U^N)$. 
By \eqref{WU}, $W(\sigma(t_1)\otimes U^N)=\sigma(t_1)U^N\otimes U^N$. 
Therefore $(t_1\ot U^N, t_2\ot U^N)\in \CSfp_N$ if and only if 
$\s(t_1)U^N=\s(t_2)$. This shows that an isomorphism between $\CSfp_N$ and 
$\T \times_{(U^N\sigma ,\sigma)} \T$ is given by 
\[                                                             \label{isom}
\CSfp_N\ni (t_1\ot U^N, t_2\ot U^N) \mapsto 
(t_1,t_2)\in \T \times_{(U^N\sigma ,\sigma)} \T. 
\]

From \eqref{i} and $\cop(U^N)=U^N\ot U^N$, it follows that 
$\cop_{\CSfp}(\iota(a))=\iota(a)\ot U^*$ and 
$\cop_{\CSfp}(\iota(b))=\iota(b)\ot U$. Hence $\iota$ is right colinear. 
Since $\iota$ is also an isometry, we can view $\OSpq$ as a 
subalgebra of $\CSfp$. Then $L_N\subset \CSfp_N$ by the 
definitions of $L_N$ and $\CSfp_N$ in \eqref{LN} and \eqref{CN}, 
respectively.
\end{proof}

We remark that the universal C*-algebra of $\OSpq$ has been studied in \cite{HMS06}, 
the K-theory of $\CSfp$ has been determined in \cite{BaumHMW}; and from the last 
example in \cite{HKMZ}, it follows that $\CSfp$ behaves well under the $\CS$-coaction 
(it is a principal Hopf-Galois extension).  

\subsection{Index computation for quantum line bundles}

The aim of this section is to illustrate that the fibre product approach may lead to a 
significant simplification of index computations. 
First we remark that, in (algebraic) quantum group theory, algebras are frequently defined 
by generators and relations similar those in \eqref{ab} for $\OSpq$ 
(more examples can be found, e.g., in \cite{KS}). 
A pair of *-representations on the same Hilbert space such that the difference yields compact 
operators gives rise to an even Fredholm module and can be used for index computations by 
pairing it with $K_0$-classes. 
If we want to compute for instance the index pairing with the $K_0$-class of the projective  
modules $L_N$ from \eqref{LN} by using the idempotents given in \eqref{EN}, then we face 
difficulties because of the growing size of the matrices. 
It is therefore desirable to find simpler representatives of K-theory classes of the 
projective modules  $L_N$. This section shows that the fibre product approach 
provides us with an effective tool for obtaining more suitable projections. 
In our example, the index pairing will reduce to its simplest possible form: 
it remains to calculate a trace of a projection onto a finite 
dimensional subspace.

We start by proving that the projective modules $\CSpq_N$ can be represented by 
elementary 1-dimensional projections. Because of the isomorphism between $\CSpq_N$ 
and $\T \times_{(U^N\sigma ,\sigma)} \T$ in 
Proposition \ref{CSfp}, this result has already been obtained in \cite{W}. 
For the convenience of the reader, we include here the proof. 
It uses essentially the same ``bra-ket'' argument that was used in \cite{HMS03,SW} to 
prove $M_N\cong \podl^{|N|+1}P_N$ for the Hopf fibration of $\SUq$.  
\begin{prop}                                                     \label{isoEchi}
For $N\in\Z$, define 
\begin{align}                                                   
&\chi_N:= (1\,,\,\fS^{|N|}\fS^{*|N|})\in \CSq,~~\mbox{for}\ \,
N <0,\\ 
&\chi_N:=(\fS^{N}\fS^{*N}\,,\,1)\in \CSq,~~\mbox{for}\ \,
N\geq 0.                                                       
\end{align}
Then the left $\CSq$-modules $\CSpq_N$ and $\CSq\chi_N$ are 
isomorphic. 
\end{prop}
\begin{proof}
Since $\s(\fS^{n}\fS^{*n})=U^{n}U^{*n}=1$ for all $n\in\N$, the projections $\chi_N$ belong to 
$\CSq=\T \times_{(\sigma ,\sigma)} \T$. 
We will use the isomorphism of  Proposition \ref{CSfp} and prove 
that $\CSq\chi_N$ is isomorphic to $\cE_{N}:=\T \times_{(U^N\sigma ,\sigma)} \T$. 

Let $N\geq 0$. 
From \eqref{CSq} and \eqref{fpLB}, it follows that 
$(f\fS^{*N}\,,\,g )\in\CSq$ for all $(f,g)\in \cE_N$. 
Therefore we can define a $\CSq$-linear map 
$\Psi_N\,:\, \cE_{N}\rightarrow \CSq\chi_N$ by 
\[                                                  \label{Psi}
\Psi_N(f,g) := (f\fS^{*N} \,,\, g)\chi_N= (f\fS^{*N} \,,\, g),
\]
where we used $\fS^*\fS=\id$ in the second equality. Since $\fS^*$ 
is right invertible, we have $(f \fS^{*N}\,,\, g)=0$
if and only if $(f \,,\, g)=0$, hence $\Psi_N$ is injective. 

Now let 
$(f,g)\chi_N\in\CSq\chi_N$. Then one has 
$(f\fS^{N},g)\in \cE_{N}$ and 
$\Psi_N(f\fS^{N},g)=(f\fS^{N}\fS^{*N},g )=(f,g)\chi_N$, 
thus $\Psi_N$ is also surjective. This proves the claim of 
Proposition \ref{isoEchi} for $N\geq 0$. The proof for $N<0$ runs
analogously with $\Psi_N$ defined by 
$\Psi_N(f,g) := (f\,,\, g\fS^{*N} )\chi_N$.
\end{proof}

Clearly, the (left) multiplication by elements of the C*-algebra $\CSq$ turns 
$L_N\cong\OSq^{|N|+1} E_N$ into a (left) $\CSq$-module. 
With a slight abuse of notation, we set 
$\CSq L_N :=\lin\{xv: x\in \CSq,\ v\in L_N\}$. 
(Later it turns out that this module is generated by one element in $L_N$ so 
that the notation is actually correct.) 
If we show that $\CSq L_N$ is isomorphic to $\CSpq_N$, 
then the elementary projections $\chi_N$ and the 
$(|N|+1)\times (|N|+1)$-matrices $E_N$ 
define the same $K_0$-class. 
The desired isomorphism will be established in the next proposition 
by using the embedding $L_N\subset \CSpq_N$ from Proposition \ref{prop1}. 

\begin{prop}                                                  \label{prop2}
 The left $\CSq$-modules $\CSq L_N\cong \CSq^{|N|+1} E_N$ and 
$\CSpq_N$ are isomorphic. 
\end{prop}

\begin{proof}
Using  embedding \eqref{i} and the inclusion from Proposition \ref{prop1}, 
we can view $\CSq L_N =\lin\{xv: x\in \CSq,\ v\in L_N\}\subset\CSpq_N $ as a 
submodule of $\CSpq_N$. Let $N\in\N_0$. It follows from the isomorphism $\Psi_N$ 
defined in \eqref{Psi} that the left $\CSq$-module 
$\CSpq_N=\{(fS^{*N},g):(f,g)\in\CSq\}$ is generated by the element $(S^{*N},1)$. 
Therefore, to prove $\CSq L_N =\CSpq_N$, it suffices to show 
that $(S^{*N},1)\in \CSq L_N$. Since $\s(z^{*n})=U^{-N}$, we have 
$(z^{*N},1)\in \T \times_{(U^N\sigma ,\sigma)} \T$. Since 
$(z^{*N},1)$ is the image  of $\iota(a^{*N})= (z^{*N}\ot U^N,1\ot U^N)$ 
under the isomorphism \eqref{isom}, 
we can view $(z^{*N},1)$ as an element of $L_N$.  
Let $t:=1-z z^*\in \T$. 
Note that $t$ is a self-adjoint operator with 
spectrum $\mathrm{spec}(t)=\{q^n: n\in\N_0\}\cup\{0\}$ (see Equation \eqref{z}). 
Applying the commutation relations \eqref{Dq}, one easily verifies that 
$z^{*N}z^{N}= \Pi_{k=1}^N(1-q^k t)$. Since $\mathrm{spec}(t)\subset [0,1]$,  
the operator $z^{*N}z^{N}$ is strictly positive. 
Hence $|z^{N}|^{-1}=(z^{*N}z^{N})^{-1/2}$ belongs to the C*-algebra $\T$. 
Moreover, $\s(|z^{N}|^{-1})=1$ since $\s(z^{*N}z^{N})=1$. 
Therefore $(|z^{N}|^{-1},1)\in\T \times_{(\sigma ,\sigma)} \T= \CSq$ and 
thus $(S^{*N},1)= (|z^{N}|^{-1},1) (z^{*N},1)\in \CSq L_N$. This completes 
the proof for $N\geq 0$. The case $N<0$ is treated analogously. 
\end{proof}

Recall that an (even) Fredholm module of an *-algebra $\A$ can be given by a pair of 
*-representations $(\rho_+,\rho_-)$  of $\A$ on a Hilbert space $\H$ such that 
the difference $\rho_+(a)-\rho_-(a)$ yields a compact operator. 
In this case, for any projection $P\in \Mat_{n, n}(\A)$, 
the operator 
$\varrho_+(P)\varrho_-(P):  \varrho_-(P)\H^{n} \rightarrow \varrho_+(P)\H^{n}$ 
is a Fredholm operator and its Fredholm index does neither depend on the 
$K_0$-class of $P$ nor on the class of $(\rho_+,\rho_-)$ in K-homology. 
This pairing between K-theory and K-homology is referred to as index pairing. 
If it happens that  $\rho_+(a)-\rho_-(a)$ yields trace class operators, 
then the index pairing can be computed by a trace formula, namely 
\begin{equation}                                                    \label{CCpair}
\langle [(\rho_+,\rho_-)],[P]\rangle=\tr_\H (\tr_{\Mat_{n,n}} (\rho_+-\rho_-)(P))
\end{equation}
In general, the computation of the traces gets more involved with increasing size of the 
matrix $P$. This will especially be the case if one works only with the polynomial 
algebras $\OSpq$ and $\OSq$, and uses the the $(|N|+1)\times(|N|+1)$-projections $E_N$ 
from \eqref{EN} with entries in belonging to $\OSq$. 
In our example, the C*-algebraic fibre product approach
improves the situation considerably since Propositions \ref{isoEchi} and \ref{prop2} 
provide us with the equivalent 1-dimensional projections $\chi_N$. 
As the index computation is one of our main objectives, we state the result in the 
following theorem. 

\begin{thm}                                           \label{T1}
Let $N\in\Z$. 
The isomorphic projective left $\CSq$-modules $\CSpq_N$, $\CSq L_N$,  
$\CSq^{|N|+1} E_N$ and $\CSq\chi_N$ define the same class 
in $K_0(\CSq)$, say $[\chi_N]$, and the pairing with the 
generators of the K-homology $K^0(\CSq)$ from the end of 
Section \ref{sec-dq} is given by 
\[                                                   \label{ic}
\<\, [(\mathrm{pr}_1,\mathrm{pr}_0)]  \,, [\chi_N]\,\>= N, \qquad
\<\, [(\pi_+\circ\hs\,,\,\pi_-\circ\hs)] \,, [\chi_N]\,\>=1.
\]
\end{thm}
\begin{proof}
The equivalences of the left $\CSq$-modules has been shown in 
Propositions \ref{isoEchi} and \ref{prop2}. In particular, we are allowed 
to choose $\chi_N$ as a representative. 

For all $N\in\Z$, the operator  
$\pi_+\circ\hs(\chi_N)-\pi_-\circ\hs(\chi_N) =\pi_+(1)-\pi_-(1)$
is the projector onto the 1-dimensional subspace $\C e_0$,  
see Equation \eqref{pi}. 
In particular, it is of trace class so that Equation \eqref{CCpair} applies. 
Since the trace of a 1-dimensional projection is 1, we get  
$$
\<\, [(\pi_+\circ\hs\,,\,\pi_-\circ\hs)] \,, [\chi_N]\,\>
=\tr_{\lN}(\pi_+(1)-\pi_-(1))=1. 
$$
Now let $N\geq 0$. 
Then $(\mathrm{pr}_1,\mathrm{pr}_0)(\chi_N)=(\mathrm{pr}_1-\mathrm{pr}_0)(S^NS^{*N},1)=1-S^NS^{*N}$ 
is the projection onto the subspace $\lin\{e_0,\ldots,e_{n-1}\}$. 
Since it is of trace class with trace equal to the dimension of its image, 
we can apply Equation \eqref{CCpair} and get
$$
\<\, [(\mathrm{pr}_1,\mathrm{pr}_0)] \,, [\chi_N]\,\>=\tr_{\lN}(1-S^NS^{*N})=N. 
$$
Analogously, for $N<0$, 
$$
\<\, [(\mathrm{pr}_1,\mathrm{pr}_0)] \,, [\chi_N]\,\>=\tr_{\lN}(S^{|N|}S^{*|N|}-1)=-|N|=N,
$$
which completes the proof. 
\end{proof}

Since the C*-algebra $\CSq$ is isomorphic to the universal C*-algebra 
of the Podle\'s spheres $\podl$, the indices in Equation \eqref{ic} 
have also been obtained in \cite{HMS03} and \cite{W}. 
In the first paper, the computations relied heavily on the index theorem, 
whereas in \cite{W} and Theorem \ref{T1} the traces were computed directly 
by using elementary projections. 

Note that Equation \eqref{ic} has a geometrical interpretation: The pairing 
with the K-homology class $[(\pi_+\circ\hs\,,\,\pi_-\circ\hs)]$ 
detects the rank of the projective 
module, and the pairing $\<\, [(\mathrm{pr}_1,\mathrm{pr}_0)]  \,, [\chi_N]\,\>= N$ 
coincides with the power of $U$ in \eqref{TT} 
and thus computes the ``winding number'', that is,   
the number of rotations of the transition function 
along the equator.  

\subsection{Equivalence to the generic Hopf fibration of quantum SU(2)}

Recall from Section \ref{qhf} that $\SUq=\oplus_{N\in\Z} M_N$, where 
$$
 M_N:=\{\, p\in \SUq\,:\, \Delta_{\SUq} (p)= p\ot U^N\,\}\cong \podl^{|N|+1} P_N
$$
with $P_N\in \Mat_{|N|+1,|N|+1}(\podl) $ given in Equation \eqref{PN}. 
For the definition of the $\OU$-coaction $\Delta_{\SUq}= (\id\ot\pr_s)\circ \Delta$, 
see Section \ref{qhf}. 
Since $\OU$ can be embedded into its universal C*-algebra, which is isomorphic to 
$\CSq$, we can turn $M_N$ into a left $\CSq$-module by considering $\ov{M}_{\!N}:=\CSq^{|N|+1} P_N$. 
It has been shown in \cite{W}, that this left $\CSq$-module 
is isomorphic to $\CSq \chi_N$, and therefore to $\CSpq_N$. 

The aim of this section is to define a left $\CSq$-module and right $\OU$-comodule $P$ 
such that, for all $N\in\Z$, the line bundle associated to the 1-di\-men\-sional left corepresentation 
${}_\C\cop(1)=U^N\ot 1$ is isomorphic to $\ov{M}_{\!N}$. A natural idea would be to consider 
the embedding of $\SUq$ into $\CSU$ and to extend the right coaction $\Delta_{\SUq}$ 
to the C*-algebra closure. But then we face the problem that $\Delta_{\SUq}$ 
is merely a coaction and not an algebra homomorphism. 
If we impose at $\OU$ the obvious multiplicative structure given by $U^NU^K=U^{N+K}$, 
and turn $\oplus_{N\in\Z} \ov{M}_{\!N}$ into a *-algebra such that the right $\OU$-coaction 
becomes an algebra homomorphism, then the C*-closure of 
$\oplus_{N\in\Z} \ov{M}_{\!N}\cong \oplus_{N\in\Z} \CSpq_N$ would be isomorphic 
to $\CSpq$ and not to $\CSU$. Note that
there cannot be an isomorphism between 
$\CSpq$ and $\CSU$ since otherwise, by the pullback diagrams \eqref{CSUq} and  \eqref{CSfp}, 
$\CS\cong \ker(\mathrm{pr_1})\cong \T\,\bar\ot\,\CS$, a contradiction. 

Instead of extending the coaction $\Delta_{\SUq}$ to some closure of $\SUq$, we turn 
$\SUq$ into a left $\CSq$-module by setting 
$P=\CSq \otimes_{\podl}\SUq$ and keeping the $\OU$-coaction, now acting on the second tensor factor. 
Then it follows immediately that 
$$
P=\mathop{\oplus}_{N\in\Z} \CSq\!\underset{\podl}{\otimes}\! M_N 
\  \ \text{and} \ \  
\CSq\underset{\podl}{\otimes} M_N=\{ p\in P\,:\, \cop_P(p)=p\ot U^N\}. 
$$
Thus our aim will be achieved if we show that 
$\ov{M}_{\!N}\cong \CSq\otimes_{\podl} M_N$. 
For this, we prove that $P$, as a left $\CSq$-module and 
right $\OU$-comodule, is isomorphic to the following fibre product 
\begin{equation}                                             \label{SUqfp}
\xymatrix{
& \makebox[48pt][c]{$
\T\, \otimes\,\OU {\underset{(\Phi\circ \pi_1,\pi_2)}{\times}} \T\, \otimes\,\OU$}
\ar[dl]_{\mathrm{pr}_1} \ar[dr]^{\mathrm{pr}_2}& \\
\T \, \otimes\,\OU\ar[d]_{\pi_1:=\sigma \otimes\id } & & \T \, \otimes\,\OU\ar[d]^{\pi_2:=\sigma \otimes\id}\\
\CS\, \otimes\,\OU \ar[rr]_{\Phi} & & \CS\, \otimes\,\OU\,. }
\end{equation}
Here $\Phi$ is defined by $\Phi(f\ot U^N)=fU^N\ot U^N$. 
Then, by comparing the pullback diagrams \eqref{CSfp} and \eqref{SUqfp} 
in the category of left $\CSq$-modules and right $\CS$-comodules, it 
follows that 
\[                                                       \label{cong}
\ov{M}_{\! N}\cong \CSpq_N\cong P\,\Box_{\CSq}\, \C\cong \CSq\otimes_{\podl} M_N
\]
with the 1-dimensional corepresentation ${}_{\CSq}\cop(1)=U^N\ot 1$ on $\C$. 
 
For simplicity of notation, we set 
$$
 \A:=\SUq,\quad \B:=\podl,\quad \bB:=\T\!\underset{(\s,\s)}{\times}\!\T\cong \CSq,\quad C:=\OU. 
$$
Recall that $\B$ can be embedded in $\A$ as well as in $\bB$, 
so both are $\B$-bimodules with respect to the multiplication. 
Moreover, the pullback diagram \eqref{TtimesT} provides us with 
*-algebra homomorphism $\pr_0:\bB\ra\T$ and $\pr_1:\bB\ra\T$ by projecting 
onto the left and right component, respectively. 
Perhaps it should here also be mentioned that $C$ is only considered as a coalgebra, 
not as an algebra. 

Let $v^N_{0},v^N_{1}, \ldots, v^N_{|N|} \in \A$  denote the matrix elements from 
the definition of $P_N$ in \eqref{PN}. Since the entries of $P_N$ belong to $\B$, 
we have $v^N_{j}v^{*N}_{k}\in\B$ for all $j,k=0,\ldots |N|$. 
The following facts are proven in  \cite[Lemma 6.5]{SW}. 
\begin{lem}                                              \label{L1}
Let $l\in\Z$ and $k,m\in \{0,\ldots |l|\}$. 
\begin{enumerate}
\item[(i)]  
For $l\geq 0$, the elements $\prr(v^{l}_{l}v^{l*}_{l})$ and 
$\prl(v^{l}_{0}v^{l*}_{0})$ are invertible in $\CT$. \smallskip
\item[(ii)]  
For $l< 0$, the elements $\prr(v^{l}_{0}v^{l*}_{0})$ and 
$\prl(v^{l}_{|l|}v^{l*}_{|l|})$ are invertible in $\CT$. \smallskip
\item[(iii)]    
$\prr(v^{l}_{k}v^{l\ast}_{l})\,
\prr(v^{l}_{l}v^{l\ast}_{l})^{-1}\,
\prr(v^{l}_{l}v^{l\ast}_{m})=\prr(v^{l}_{k}v^{l\ast}_{m})$ \ \,and 
\smallskip\\
$\prl(v^{l}_{k}v^{l\ast}_{0})
\prl(v^{l}_{0}v^{l\ast}_{0})^{-1}
\prl(v^{l}_{0}v^{l\ast}_{m})=\prl(v^{l}_{k}v^{l\ast}_{m})$\ \,for $l\geq 0$.\smallskip
\item[(iv)]
$\prr(v^{l}_{k}v^{l\ast}_{0})\,
\prr(v^{l}_{0}v^{l\ast}_{0})^{-1}\,
\prr(v^{l}_{0}v^{l\ast}_{m})=\prr(v^{l}_{k}v^{l\ast}_{m})$\ \,and 
\smallskip\\
$\prl(v^{l}_{k}v^{l\ast}_{|l|})
\prl(v^{l}_{|l|}v^{l\ast}_{|l|})^{-1}
\prl(v^{l}_{|l|}v^{l\ast}_{m})=\prl(v^{l}_{k}v^{l\ast}_{m})$\ \,for $l< 0$, 
\end{enumerate}
\end{lem}

We can turn $\T$ into a $\bB$-bimodule by setting
$a.t. b := \prl(a)\,t\,\prl(b)$ and 
$a.t. b = \prr(a)\,t\,\prr(b)$, where 
$a,b\in \bB$ and $t\in\T$.
To distinguish between both bimodules, we denote $\T$ equipped 
with the first action by $\T_-$, and write $\T_+$ if we use 
the second action. 
Clearly, as left or right $\CB$-module, both are generated by 
$1\in\T$. 
The next proposition is the key in proving \eqref{cong}.

\begin{prop}                                                  \label{TBM}
The left $\CB$-modules $\T_\pm$ and $\T_\pm\,\otimes_\B\, M_l$ are isomorphic. 
The corresponding isomorphisms are given by 
\begin{align*}
&\psi_{l,+}:\T_+\ra\T_+\otimes_\B M_l,\quad 
\psi_{l,+}(t)= t\, \prl(v^{l}_{0}v^{l*}_{0})^{-1/2}\otimes_\B v^{l}_{0}, 
\quad l\geq 0,\\
&\psi_{l,-}:\T_-\ra\T_-\otimes_\B M_l,\quad 
\psi_{l,-}(t)= t\, \prr(v^{l}_{l}v^{l*}_{l})^{-1/2}\otimes_\B v^{l}_{l}, 
\quad l\geq 0,\\
&\psi_{l,+}:\T_+\ra\T_+\otimes_\B M_l,\quad 
\psi_{l,+}(t)= t\, \prl(v^{l}_{|l|}v^{l*}_{|l|})^{-1/2}\otimes_\B v^{l}_{|l|}, 
\quad l< 0,\\
&\psi_{l,-}:\T_-\ra\T_-\otimes_\B M_l,\quad 
\psi_{l,-}(t)= t\, \prr(v^{l}_{0}v^{l*}_{0})^{-1/2}\otimes_\B v^{l}_{0}, 
\quad l< 0.
\end{align*}
The inverse isomorphisms satisfy, for all $k=0,1,\dots,|l|$, 
\begin{align}                                                    \label{psi+}
 &  \psi_{l,+}^{-1} (1\otimes_\B v^{l}_{k}) = \prl(v^{l}_{k}v^{l\ast}_{0})
\prl(v^{l}_{0}v^{l\ast}_{0})^{-1/2}, \quad l\geq 0,\\
 & \psi_{l,-}^{-1} (1\otimes_\B v^{l}_{k})
=\prr(v^{l}_{k}v^{l\ast}_{l})
\prr(v^{l}_{l}v^{l\ast}_{l})^{-1/2}, \quad l\geq 0,\\
&  \psi_{l,+}^{-1} (1\otimes_\B v^{l}_{k}) = \prr(v^{l}_{k}v^{l\ast}_{|l|})\,
\prr(v^{l}_{|l|}v^{l\ast}_{|l|})^{-1/2}, \quad l< 0,\\
 & \psi_{l,-}^{-1} (1\otimes_\B v^{l}_{k})=\prr(v^{l}_{k}v^{l\ast}_{0})\,
\prr(v^{l}_{0}v^{l\ast}_{0})^{-1/2}, \quad l<0.
\end{align}
\end{prop}
\begin{proof}
We prove the proposition for $\psi_{l,+}$ with $l\geq 0$,  
the other cases are treated analogously. 
Since $\prl(v^{l}_{0}v^{l*}_{0})$ is positive and invertible, 
$\prl(v^{l}_{0}v^{l*}_{0})^{-1/2}\in\T$ is invertible, and thus $\psi_{l,+}$ is 
injective. The left $\CB$-module $\T_+\,\otimes_\B\, M_l$ is 
generated by $1\otimes_\B v^{l}_{k}$, $k=0,1,\dots,l$ 
(cf.\ \cite[Theorem 4.1]{SW}). 
As $\psi_{l,+}$ is left $\CB$-linear,  
it suffices to prove that the elements  $1\otimes_\B v^{l}_{k}$ belong 
to the image of $\psi_{l,+}$. Applying \eqref{1} and Lemma \ref{L1}(iii), 
we get 
\begin{align*}                                                  \nonumber
1\otimes_\B v^{l}_{k}&=\mbox{$\sum$}_j 1\otimes_\B v^{l}_{k} v^{l*}_{j} v^{l}_{j} 
=\mbox{$\sum$}_j \prl(v^{l}_{k} v^{l*}_{j} )\otimes_\B  v^{l}_{j} \\  \nonumber
&=\mbox{$\sum$}_j \prl(v^{l}_{k}v^{l\ast}_{0})
\prl(v^{l}_{0}v^{l\ast}_{0})^{-1}
\prl(v^{l}_{0} v^{l*}_{j})\otimes_\B  v^{l}_{j}\\                   \nonumber
&= \mbox{$\sum$}_j \prl(v^{l}_{k}v^{l\ast}_{0})
\prl(v^{l}_{0}v^{l\ast}_{0})^{-1}
\otimes_\B v^{l}_{0} v^{l*}_{j}v^{l}_{j}\\                         \nonumber
&= \prl(v^{l}_{k}v^{l\ast}_{0})
\prl(v^{l}_{0}v^{l\ast}_{0})^{-1}
\otimes_\B v^{l}_{0} 
=\psi_{l,+}\big(\prl(v^{l}_{k}v^{l\ast}_{0})
\prl(v^{l}_{0}v^{l\ast}_{0})^{-1/2}\big).
\end{align*}
This proves the surjectivity of  $\psi_{l,+}$ and Equation \eqref{psi+}. 
\end{proof}

Using the last proposition and the decomposition $\A= \oplus_{N\in\Z} M_N$, we can define 
left $\bB$-linear, right $C$-colinear isomorphisms 
$$
\Psi_- : \T_-\otimes_\B \A \ra \mathop{\oplus}_{N\in\Z}\T_-\otimes U^N,\quad 
\Psi_+ : \T_+\otimes_\B \A \ra \mathop{\oplus}_{N\in\Z}\T_+\otimes U^N
$$
by setting 
\[                                                                      \label{tm}
\Psi_\pm (t\otimes_\B m_N)= \psi_{N,\pm}^{-1}(t\otimes_\B m_N)\,\otimes\, U^N,\qquad 
t\in\T,\ \,m_N\in M_N. 
\]
Next we define left $\bB$-linear, right $C$-colinear surjections 
$$\Prpm : \CB\otimes_\B \A \ra  \T_\pm\otimes_\B \A,
$$ 
by 
\[                                                    \label{prpm}
\pr_-((t_1,t_2)\otimes_\B a) :=  t_1\otimes_\B a, \quad 
 \pr_+((t_1,t_2)\otimes_\B a) :=  t_2\otimes_\B a. 
\]
Furthermore, we turn $\CS$ into a left $\bB$-module by defining $b.f:=\s(b)f$ for all $b\in\bB$ 
and $f\in \CS$. Now consider the following diagram in the category of 
left $\bB$-modules, right $C$-comodules: 
\[                                              \label{cd} 
\xymatrix{                 
\CB\otimes_\B \A \ar[d]_{\Psi_-\circ\pr_-}    \ar[rr]^{\Psi_+\circ\pr_+ } &
 & \T_+\otimes C \ar[d]^{\s\otimes \id} \\
\T_-\otimes C 
\ar[rr]_{\Phi\circ(\s\otimes \id)}&  &\CS\otimes C,
}
\]
where $\Phi$ is the same as in \eqref{SUqfp}. 

\begin{lem}                                  \label{L2}
 The diagram \eqref{cd} is commutative, $\Psi_-\circ\pr_-$ and $\Psi_+\circ\pr_+ $ 
are sur\-jec\-tive and 
$\ker( \Psi_-\circ\pr_-)\cap \ker(\Psi_+\circ\pr_+  )=\{0\}$. 
\end{lem}

\begin{proof}
Since all maps are left $\CB$-linear, it suffices to prove the lemma  
for generators of the left $\CB$-module $\CB\otimes_\B \A$. Moreover, 
since $\A=\oplus_{N\in\Z}M_N$, we can restrict ourselves to the 
generators of the left $\B$-modules $M_N$. 

Let $l\geq 0$. Since $\s(\pr_0(f))=\s(f)=\s(\pr_1(f))$ for all $f\in\CB$ by \eqref{CSq}, 
we get from Equation \eqref{tm} and Lemma \ref{L1}
\begin{align}                                                            \label{RX}
 (\hs\otimes \id)\circ \Psi_+\circ\Prl \big(1\otimes_\B v^{l}_{k} \big) 
 &=  \hs(v^{l}_{k}v^{l\ast}_{0})
\hs(v^{l}_{0}v^{l\ast}_{0})^{-1/2} \otimes U^l,\\                        \label{L}
\phi\circ(\hs\otimes \id)\circ \Psi_-\circ\Prr\big(1\otimes_\B v^{l}_{k} \big) 
&= \hs(v^{l}_{k}v^{l\ast}_{l})
\hs(v^{l}_{l}v^{l\ast}_{l})^{-1/2}U^l\otimes U^l.
\end{align}
By Lemma \ref{L1} (iii) (with $m=0$), we have  
\[                                                                       \label{k0}
\hs(v^{l}_{k}v^{l\ast}_{0})=
\hs(v^{l}_{k}v^{l\ast}_{l})\,
\hs(v^{l}_{l}v^{l\ast}_{l})^{-1}\,
\hs(v^{l}_{l}v^{l\ast}_{0}).
\]
Inserting the latter equation into \eqref{RX} and comparing with \eqref{L} shows that 
it suffices to prove 
\[                                                                     \label{ul}
\hs(v^{l}_{l}v^{l\ast}_{l})^{-1/2}\,
\hs(v^{l}_{l}v^{l\ast}_{0})\,\hs(v^{l}_{0}v^{l\ast}_{0})^{-1/2} = U^l.
\]
It follows from \cite[Lemma 2.2]{W} 
(with $v^{l}_{l}\sim u_l$ and $v^{l}_{0}\sim w_l$), 
or can be computed directly by using explicit expressions for $v^{l}_{0}$ and $v^{l}_{l}$, that 
$v^{l}_{l}v^{l\ast}_{0}\sim \eta_s^l$. 
From the embedding \eqref{etaT}, we deduce that $\eta_s^l$ has polar decomposition 
$\eta_s^l=(S^l,S^l) |\eta_s^l|$.  Therefore we can write 
$v^{l}_{l}v^{l\ast}_{0}=  (S^l,S^l) |v^{l}_{l}v^{l\ast}_{0}|$ which implies 
$$
\hs(v^{l}_{l}v^{l\ast}_{0})= \hs(|v^{l}_{l}v^{l\ast}_{0}|)\, U^l. 
$$
By comparing with \eqref{ul}, we see that it now suffices to verify 
\[                                                                    \label{|v|}
\hs(v^{l}_{l}v^{l\ast}_{l})^{-1/2}\,
\hs(|v^{l}_{l}v^{l\ast}_{0}|)\,\hs(v^{l}_{0}v^{l\ast}_{0})^{-1/2} = 1. 
\]
Multiplying both sides of  Equation \eqref{k0} with $\hs(v^{l}_{l}v^{l\ast}_{l})$ gives 
$$
\hs(v^{l}_{0}v^{l\ast}_{0})\,
\hs(v^{l}_{l}v^{l\ast}_{l})=
\hs(v^{l}_{0}v^{l\ast}_{l})\,
\hs(v^{l}_{l}v^{l\ast}_{0}).
$$
Thus
$$
\hs(|v^{l}_{l}v^{l\ast}_{0}|)=
\hs\big((v^{l}_{0}v^{l\ast}_{l}v^{l}_{l}v^{l\ast}_{0})^{1/2}\big)
=\big(\hs(v^{l}_{0}v^{l\ast}_{l})\hs(v^{l}_{l}v^{l\ast}_{0})\big)^{1/2}
=\big(\hs(v^{l}_{0}v^{l\ast}_{0})
\hs(v^{l}_{l}v^{l\ast}_{l}) \big)^{1/2},
$$
which proves \eqref{|v|}. This concludes the proof of the commutativity 
of \eqref{cd} for $l\geq 0$. The case $l<0$ is treated 
analogously. 

The surjectivity of $\Psi_-\circ\pr_-$ and $\Psi_+\circ\pr_+ $ follows from the bijectivity of $\Psi_\pm$ 
and the surjectivity of $\pr_\pm$. 

Suppose that 
$\sum_{k=1}^n(r_k,s_k)\ot_B a_k\in \ker( \Psi_-\circ\pr_-)\cap \ker(\Psi_+\circ\pr_+)$. 
Since $\Psi_\pm$ is an isomorphism, we get 
$\sum_{k=1}^n r_k\ot_B a_k=0$ and $\sum_{k=1}^n s_k\ot_B a_k=0$ by \eqref{prpm}. 
Hence $\sum_{k=1}^n(r_k,s_k)\ot_B a_k=
\sum_{k=1}^n (r_k,0)\ot_B a_k+\sum_{k=1}^n (0,s_k)\ot_B a_k=0$ which 
proves last claim of the lemma. 
\end{proof}

We are now in a position to prove the main theorem of this section. 
\begin{thm}
 There is an isomorphism of left $\CSq$-modules and right $\OU$-comodules 
between the fibre product 
$\T\, \otimes\,\OU {\times}_ {(\Phi\circ \pi_1,\pi_2)}\T\, \otimes\,\OU$ from \eqref{SUqfp} and 
$\CSq \otimes_{\podl}\SUq$. 
Moreover, the chain of isomorphisms in \eqref{cong} holds. 
\end{thm}
\begin{proof}
Lemma \ref{L2} states that  $\CSq \otimes_{\podl}\SUq$ is an universal object of the 
pull back diagram \eqref{cd}. Comparing \eqref{cd} and \eqref{SUqfp} shows that 
both pullback diagrams define up to isomorphism the same universal object 
which proves the first part of  the theorem. 

The first isomorphism in \eqref{cong} follows from the Murray-von Neumann equivalence 
of the corresponding projections, see \cite{W}.  
The second isomorphism follows from the above equivalence of pullback diagrams, 
and the last one from fact that all mappings in \eqref{cd} are right 
$\OU$-colinear. 
\end{proof}

\subsection*{Acknowledgment}
The author thanks Piotr M.\ Hajac for many discussions on the subject. This work was financially supported by the CIC of the Universidad Michoacana 
and European Commission grant PIRSES-GA-2008-230836.


\end{document}